\documentclass{svmult}

\usepackage{graphicx}
\usepackage{amsfonts}
\usepackage{amsmath}
\usepackage{amssymb}
\usepackage[ansinew]{inputenc}
\usepackage{latexsym}
\usepackage{tabularx}
\usepackage{epsfig}
\usepackage{mathrsfs}
\usepackage[active]{srcltx}

\allowdisplaybreaks

\newtheorem{algorithm}{Algorithm}

\newcommand{\abs}[1]{\left\vert#1\right\vert}

\newcommand{\To}{\rightarrow}

\newcommand{\bsgamma}{\boldsymbol{\gamma}}
\newcommand{\bsalpha}{\boldsymbol{\alpha}}

\newcommand{\bsk}{\boldsymbol{k}}
\newcommand{\bsl}{\boldsymbol{l}}

\newcommand{\bsh}{\boldsymbol{h}}

\newcommand{\bsx}{\boldsymbol{x}}
\newcommand{\bsz}{\boldsymbol{z}}
\newcommand{\bsg}{\boldsymbol{g}}

\newcommand{\bss}{\boldsymbol{s}}

\newcommand{\bsy}{\boldsymbol{y}}

\newcommand{\cP}{{\cal P}}

\newcommand{\wal}{{\rm wal}}

\newcommand{\trig}{{\rm trig}}

\newcommand{\icomp}{\mathtt{i}}

\newcommand{\rd}{\,\mathrm{d}}
\newcommand{\FF}{\mathbb{F}}

\newcommand{\Natural}{\mathbb{N}}
\newcommand{\Integer}{\mathbb{Z}}
\newcommand{\Real}{\mathbb{R}}
\newcommand{\Complex}{\mathbb{C}}

\newcommand{\cS}{{\cal S}}

\newcommand{\ZZ}{\mathbb{Z}}

\newcommand{\ee}{\mathrm{e}}
\newcommand{\cH}{{\cal H}}

\renewenvironment{proof}{\begin{trivlist}
    \item[\hskip\labelsep{\bf Proof.}]}{$\hfill\Box$\end{trivlist}}

\begin{document}
\title*{Tractability of multivariate integration in hybrid function spaces}
\author{Peter
Kritzer\thanks{P.~Kritzer is supported by 
the Austrian Science Fund (FWF), Projects P23389-N18 and F05506-26. The latter is part of the Special Research Program 
``Quasi-Monte Carlo Methods: Theory and Applications''.} \and Friedrich
Pillichshammer\thanks{F.~Pillichshammer is supported by
the Austrian Science Fund (FWF) Project F5509-N26, which is part of the Special Research Program 
``Quasi-Monte Carlo Methods: Theory and Applications''.}}

\titlerunning{Tractability of integration in hybrid spaces}

\institute{Department of Financial Mathematics, Johannes Kepler University
Linz, Altenbergerstr. 69, 4040 Linz, Austria. \email{\{peter.kritzer, friedrich.pillichshammer\}(AT)jku.at}}

\maketitle

\abstract{We consider tracatability of integration in reproducing kernel 
Hilbert spaces which are a tensor product of a Walsh space and a Korobov space. 
The main result provides necessary and sufficient conditions for weak, polynomial and strong polynomial tractability.}

\section{Introduction}

In this paper we study multivariate integration $$I_s(f)=\int_{[0,1]^s} f(\bsx) \rd \bsx$$ in 
reproducing kernel Hilbert spaces $\cH(K)$ of functions $f:[0,1]^s \rightarrow \Real$, equipped with the norm $\|\cdot \|_{\cH(K)}$, 
where $K$ denotes the reproducing kernel. 
We refer to Aronszajn~\cite{A50} for an introduction to the theory of reproducing kernel Hilbert spaces. 
Without loss of generality, see, e.g., \cite{NW08,TWW88}, we can restrict ourselves
to approximating $I_s(f)$ by means of {\it linear algorithms} $Q_{N,s}$ of the form
\[Q_{N,s}(f,\cP):=\sum_{k=0}^{N-1} q_k f(\bsx_k),\]
with coefficients $q_k\in \Complex$ and deterministically chosen sample points $\bsx_k\in[0,1)^s$.
In this paper we further restrict ourselves to considering only $q_k$ of the form $q_k=1/N$ for all
$0 \le k < N$ in which case one speaks of {\it quasi-Monte Carlo (QMC)
algorithms}. QMC algorithms are often used in practical applications especially if $s$ is large.
We are interested in studying the {\it worst-case integration error},
$$
e(\cH(K),\cP)=\sup_{f \in \cH(K) \atop
\|f\|_{\cH(K)} \le 1} \abs{I_s(f)-Q_{N,s}(f,\cP)}.$$ 
Let $e(N,s)$ be the {\it $N$th minimal QMC worst-case error},
$$
e(N,s)=\inf_{\cP}e(\cH(K),\cP),
$$
where the infimum is extended over all $N$-element point sets in $[0,1)^s$. The initial error $e(0,s)$ is defined as $$e(0,s)=\sup_{f \in \cH(K) \atop
\|f\|_{\cH(K)} \le 1} \abs{I_s(f)}$$ and is used as a reference value. 
In this paper we are interested in the dependence of the worst-case error on the dimension $s$. 
To study this dependence systematically we consider the so-called {\it information complexity} defined as 
$$N_{\min}(\varepsilon,s)=\min\{N \in \Natural \ : \ e(N,s) \le \varepsilon \, e(0,s)\},$$ 
which is the minimal number of points required to reduce the initial error by a factor of $\varepsilon$.

We would like to avoid cases where the information complexity $N_{\min}(\varepsilon,s)$
grows exponentially or even faster with the dimension $s$ or with $\varepsilon^{-1}$. 
To quantify the behavior of the information complexity we use the following notions of {\it tractability}.

We say that the integration problem in $\cH(K)$ is 
\begin{itemize}
 \item {\it weakly QMC-tractable}, if $$\lim_{s+\varepsilon^{-1} \rightarrow \infty} \frac{\log N_{\min}(\varepsilon,s)}{s+\varepsilon^{-1}}=0;$$
 \item {\it polynomially QMC-tractable}, if there exist non-negative numbers $c,p$, 
  and $q$ such that $$N_{\min}(\varepsilon,s) \le c s^q \varepsilon^{-p};$$
 \item {\it strongly polynomially QMC-tractable}, if there exist non-negative numbers 
 $c$ and $p$ such that $$N_{\min}(\varepsilon,s) \le c \varepsilon^{-p}.$$
\end{itemize}
Of course, strong polynomial QMC-tractability implies polynomial QMC-tractability which in turn implies weak QMC-tractability. 
If we do not have weak QMC-tractability, then we say that the integration problem in $\cH(K)$ is {\it intractable}.

In the existing literature, many authors have studied tractability (since we only deal with QMC-rules here we omit the prefix 
``QMC'' from now on) of integration in many different reproducing kernel Hilbert spaces. 
The current state of the art of tractability theory is summarized in the three volumes of the book of 
Novak and Wo\'{z}niakowski \cite{NW08,NW10,NW12} 
which we refer to for extensive information on this subject and further references. Most of these investigations have in common that 
reproducing kernel Hilbert spaces 
are tensor products of one-dimensional spaces whose kernels are all of the same type (but maybe equipped with different weights). 
In this paper we consider the case where the reproducing kernel is a tensor product of spaces with kernels of different type. 
We call such spaces {\it hybrid spaces}. As far as we are aware of, this problem has not been studied in the literature so far. 
This paper is a first attempt in this direction.  

In particular, we consider the tensor product of an $s_1$-dimensional weighted Walsh space and an $s_2$-dimensional weighted Korobov 
space (the exact definitions will be given in the next section). The study of such spaces could be important in view of the integration 
of functions which are periodic with respect to some of the components and, for example, piecewise constant with respect to the remaining components. 
Moreover, it has been pointed out by several scientists (see, e.g., \cite{K13,L14}) that hybrid integration problems may be relevant for 
certain integration problems in applications.

From the analytical point of view, it is very challenging to deal with integration in hybrid spaces. 
The reason for this is the rather complex interplay between the different analytic and algebraic structures of the kernel functions. 
In the present study we are concerned with Fourier analysis carried out simultaneously with respect to the Walsh and the trigonometric function system. 
The problem is also closely related to the study of {\it hybrid point sets} which received much attention in recent times 
(see, for example, \cite{H11,HellK11,HK11,HKLP09,HL11,K11,KLP13,KP13}).

The paper is organized as follows. In Section~\ref{sec_space} we introduce the Hilbert space under consideration 
in this paper. The main result states necessary and sufficient conditions for various notions of tractability and 
is stated in Section~\ref{sec_main_rsult}. In Section~\ref{pr_nec} we prove the 
necessary conditions and in Section~\ref{pr_suf} the sufficient ones.

\section{The Hilbert space}\label{sec_space}

\subsection{Basic Notation}

Let $k\in\Natural_0$ with $b$-adic representation
$k=\sum_{i=0}^\infty \kappa_i b^i$, $\kappa_i \in
\{0,\ldots,b-1\}$. Furthermore, let $x\in[0,1)$ with $b$-adic
representation $x = \sum_{i=1}^\infty \xi_i b^{-i}$,
$\xi_i\in\{0,\ldots,b-1\}$, unique in the sense that infinitely many
of the $\xi_i$ differ from $b-1$. If $\kappa_a\ne 0$ is the most significant
nonzero digit of $k$, we define the $k$th {\it Walsh function} $\wal_{k}:[0,1)
\To \Complex$ (in base $b$) by
\[\wal_{k}(x):=
  \ee \left(\frac{\xi_1 \kappa_0 + \cdots + \xi_{a+1} \kappa_a}{b}\right),\]
where $\ee(v):=\exp(2\pi\icomp v)$. For dimension $s \ge 2$ and
vectors $\bsk = (k_1, \ldots, k_{s}) \in \Natural_0^{s}$ and $\bsx =
(x_1,\ldots,x_s) \in [0,1)^{s}$ we define the $\bsk$th Walsh function $\wal_{\bsk} : [0,1)^{s} \To
\Complex$ by $\wal_{\bsk}(\bsx):=\prod_{j=1}^{s} \wal_{k_j}(x_j)$.

\medskip

Furthermore, for $\bsl\in \Integer^{s}$ and $\bsy \in \Real^{s}$ we define the $\bsl$th trigonometric function by 
$\ee_{\bsl}(\bsy):=\ee(\bsl \cdot \bsy)$, where ``$\cdot$'' denotes the usual dot product.
\medskip

We define two functions $r^{(1)}, r^{(2)}$: let $\alpha>1$ and $\gamma>0$ be reals and let $\bsgamma=(\gamma_j)_{j \ge 1}$ be a sequence of positive reals.
\begin{itemize}
\item For integer $b\ge 2$, and $k\in \Natural_0$, $\alpha>1$ let
\[r_{\alpha,\gamma}^{(1)}(k):=\begin{cases}
                        1 & \mbox{if $k=0$,}\\
		        \gamma b^{-\alpha \psi_b (k)} & \mbox{if $k\neq 0$, where $\psi_b (k):=\lfloor \log_b k \rfloor$.}
                       \end{cases}
\]
For $\bsk=(k_1,\ldots,k_{s})\in\Natural_0^{s}$ let $r^{(1)}_{\alpha,\bsgamma}(\bsk):=\prod_{j=1}^{s} r_{\alpha,\gamma_j} (k_j)$.
Even though the parameter $b$ occurs in the definition of $r_{\alpha,\gamma}^{(1)}$, we do not explicitly include it in our notation 
as the choice of $b$ will usually be clear from the context.
\item For $l\in\Integer$ let
\[r^{(2)}_{\alpha,\gamma}(l):=\begin{cases}
                        1 & \mbox{if $l=0$,}\\
		        \gamma\abs{l}^{-\alpha} & \mbox{if $l\neq 0$}.
                       \end{cases}
\]
For $\bsl=(l_1,\ldots,l_{s})\in\Integer^{s}$ let $r^{(2)}_{\alpha,\bsgamma}(\bsl):=\prod_{j=1}^{s} r_{\alpha,\gamma_j} (l_j)$.
\end{itemize}

\subsection{Definition of the Hilbert space}

Let $s_1,s_2\in \Natural_0$ such that $s_1+s_2\ge 1$. We write $\bss=(s_1,s_2)$. 
For $\bsx=(x_1,\ldots,x_{s_1})\in [0,1)^{s_1}$ and $\bsy=(y_1,\ldots,y_{s_2})\in[0,1)^{s_2}$,
we use the short hand $(\bsx,\bsy)$ for $(x_1,\ldots,x_{s_1},y_1,\ldots,y_{s_2})\in [0,1)^{s_1+s_2}$.  

Let $\bsgamma^{(1)}=(\gamma_j^{(1)})_{j \ge 1}$ and $\bsgamma^{(2)}=(\gamma_j^{(2)})_{j \ge 1}$ be
non-increasing sequences of positive real numbers. We write $\bsgamma$ for the tuple $(\bsgamma^{(1)},\bsgamma^{(2)})$. 
Furthermore, let $\alpha_1,\alpha_2\in\Real$, with $\alpha_1,\alpha_2>1$ and write $\bsalpha=(\alpha_1,\alpha_2)$.

We first define a function $K_{\bss,\bsalpha,\bsgamma}:[0,1]^{s_1+s_2}\times[0,1]^{s_1+s_2}\To \Complex$ 
(which will be the kernel function of a Hilbert space, as we shall see later) by
\begin{eqnarray*}
\lefteqn{K_{\bss,\bsalpha,\bsgamma}((\bsx,\bsy),(\bsx',\bsy'))}\\
& := &\sum_{\bsk\in\Natural_0^{s_1}}
  \sum_{\bsl\in\Integer^{s_2}}
  r^{(1)}_{\alpha_1,\bsgamma^{(1)}}(\bsk) r^{(2)}_{\alpha_2,\bsgamma^{(2)}}(\bsl) 
  \wal_{\bsk}(\bsx) \ee_{\bsl} (\bsy)\overline{\wal_{\bsk}(\bsx') \ee_{\bsl} (\bsy')}
\end{eqnarray*}
for $(\bsx,\bsy),(\bsx',\bsy')\in [0,1]^{s_1+s_2}$ (to be more precise, we should write $\bsx,\bsx'\in[0,1]^{s_1}$ and $\bsy,\bsy'\in[0,1]^{s_2}$; 
from now on, when using the notation $(\bsx,\bsy)\in [0,1]^{s_1+s_2}$, we shall always tacitly assume that $\bsx\in[0,1]^{s_1}$ and 
$\bsy\in[0,1]^{s_2}$).

Note that $K_{\bss,\bsalpha,\bsgamma}$ can be written as
\begin{equation}\label{eqkernelproduct}
K_{\bss,\bsalpha,\bsgamma}((\bsx,\bsy),(\bsx',\bsy'))=
K_{s_1,\alpha_1,\bsgamma^{(1)}}(\bsx,\bsx')K_{s_2,\alpha_2,\bsgamma^{(2)}}(\bsy,\bsy'),
\end{equation}
where $K_{s_1,\alpha_1,\bsgamma^{(1)}}$ is the reproducing kernel of a Hilbert space based on Walsh 
functions. This space is defined as 
\[\cH(K_{s_1,\alpha_1,\bsgamma^{(1)}}):=
\left\{f=\sum_{\bsk \in \Natural_0^{s_1}} \widehat{f}_{\wal}(\bsk) \wal_{\bsk} \, : \, \|f\|_{s_1,\alpha_1,\bsgamma^{(1)}} < \infty\right\},\]
where the $\widehat{f}_{\wal}(\bsk):= \int_{[0,1]^{s_1}} f(\bsx)  \overline{\wal_{\bsk}(\bsx)} \rd \bsx$ are the 
Walsh coefficients of $f$ and where $\|f\|_{s_1,\alpha_1,\bsgamma^{(1)}}=\langle f, f \rangle_{s_1,\alpha_1,\bsgamma^{(1)}}^{1/2}$, 
with the inner product   
\[\langle f, g \rangle_{s_1,\alpha_1,\bsgamma^{(1)}} 
= \sum_{\bsk \in \Natural_0^{s_1}} (r_{\alpha_1,\bsgamma^{(1)}}^{(1)}(\bsk))^{-1} \widehat{f}_{\wal}(\bsk)  \overline{\widehat{g}_{\wal}(\bsk)}.\] 
This so-called {\it Walsh space} was first introduced and studied in \cite{DP05}.

Furthermore, $K_{s_2,\alpha_2,\bsgamma^{(2)}}$ is the reproducing kernel of a Hilbert space based on trigonometric
functions. This second function space is defined as
\[\cH(K_{s_2,\alpha_2,\bsgamma^{(2)}}):=
\left\{f=\sum_{\bsl \in \Integer_0^{s_2}} \widehat{f}_{\trig}(\bsl) \ee_{\bsl} \, : \, \|f\|_{s_2,\alpha_2,\bsgamma^{(2)}} < \infty\right\},\]
where the $\widehat{f}_{\trig}(\bsl):= \int_{[0,1]^{s_2}} f(\bsy)  \overline{\ee_{\bsl}(\bsy)} \rd \bsy$ are the 
Fourier coefficients of $f$ and where $\|f\|_{s_2,\alpha_2,\bsgamma^{(2)}}=\langle f, f \rangle_{s_2,\alpha_2,\bsgamma^{(2)}}^{1/2}$, 
with the inner product   
\[\langle f, g \rangle_{s_2,\alpha_2,\bsgamma^{(2)}} 
= \sum_{\bsh \in \Integer^{s_2}} (r_{\alpha_2,\bsgamma^{(2)}}^{(2)}(\bsl))^{-1} \widehat{f}_{\trig}(\bsk)  \overline{\widehat{g}_{\trig}(\bsk)}.\] This so-called {\it Korobov space} is studied in many papers. We refer to \cite{NW10,SW01} and the references therein for further information. 

Furthermore, \cite[Part~I, Section~8, Theorem~I, p.~361]{A50} implies that $K_{\bss,\bsalpha,\bsgamma}$ 
is the reproducing kernel of the tensor product of 
the spaces $\cH(K_{s_1,\alpha_1,\bsgamma^{(1)}})$ and $\cH(K_{s_2,\alpha_2,\bsgamma^{(2)}})$, 
i.e., of the space 
\[\cH(K_{\bss,\bsalpha,\bsgamma})=\cH(K_{s_1,\alpha_1,\bsgamma^{(1)}}) \otimes \cH(K_{s_2,\alpha_2,\bsgamma^{(2)}}).\]
The elements of $\cH(K_{\bss,\bsalpha,\bsgamma})$ are defined on $[0,1]^{s_1+ s_2}$, and the space is
equipped with the inner product
\[
  \langle f, g \rangle_{\bss,\bsalpha,\bsgamma}
= \sum_{\bsk\in\Natural_0^{s_1}}\sum_{\bsl \in \Integer_0^{s_2}}
  (r^{(1)}_{\alpha_1,\bsgamma^{(1)}}(\bsk))^{-1} (r^{(2)}_{\alpha_2,\bsgamma^{(2)}}(\bsl))^{-1} 
\widehat{f}(\bsk,\bsl)  \overline{\widehat{g}(\bsk,\bsl)},
\]
where
\[
   \widehat{f}(\bsk,\bsl) := \int_{[0,1]^{s_1+s_2}} f(\bsx,\bsy)  \overline{\wal_{\bsk}(\bsx) \ee_{\bsl}(\bsy)} \rd \bsx \rd \bsy.
\]
The norm in $\cH(K_{\bss,\bsalpha,\bsgamma})$ is given by $||f||_{\bss,\bsalpha,\bsgamma} := \langle f, f \rangle_{\bss,\bsalpha,\bsgamma}^{1/2}$. 

\medskip 

We study the problem of numerically integrating a function $f\in \cH(K_{\bss,\bsalpha,\bsgamma})$, i.e., we would like
to approximate
\[I_{\bss} (f)=\int_{[0,1]^{s_1}}\int_{[0,1]^{s_2}} f(\bsx,\bsy)\rd \bsx\rd\bsy.\]
We use a QMC rule based on a point set $\cS_{N,\bss}=((\bsx_n,\bsy_n))_{n=0}^{N-1}\subseteq [0,1)^{s_1+s_2}$, 
so we approximate $I_{\bss} (f)$ by 
\[\frac{1}{N}\sum_{n=0}^{N-1} f(\bsx_n,\bsy_n).\]
Using \cite[Proposition~2.11]{DP10} we obtain that the squared worst-case integration error of functions from $\cH(K_{\bss,\bsalpha,\bsgamma})$ 
is given by 
\begin{equation}\label{wcerror}
e^2(\cH(K_{\bss,\bsalpha,\bsgamma}),\cS_{N,\bss})=-1+\frac{1}{N^2} \sum_{n,n'=0}^{N-1} 
K_{\bss,\bsalpha,\bsgamma}((\bsx_n,\bsy_n),(\bsx_{n'},\bsy_{n'}))
\end{equation}
and the initial error $e(0,s_1+s_2)$ is one for all $s_1$, $s_2$, which means that the integration problem is well normalized.

\section{The main result}\label{sec_main_rsult}

The main result of this paper states necessary and sufficient conditions for the various notions of tractability.
\begin{theorem}\label{thm_main}
We have strong polynomial QMC-tractability of multivariate integration in $\cH(K_{\bss,\bsalpha,\bsgamma})$ iff 
\begin{equation}\label{eqcondstrongtract}
 \lim_{(s_1 +s_2)\To\infty}\left(\sum_{j=1}^{s_1}\gamma_j^{(1)} + \sum_{j=1}^{s_2}\gamma_j^{(2)}\right)<\infty.
\end{equation}
We have polynomial QMC-tractability of multivariate integration in $\cH(K_{\bss,\bsalpha,\bsgamma})$ iff 
\begin{equation}\label{eqcondtract}
 \lim_{(s_1 +s_2)\To\infty}\left(\frac{\sum_{j=1}^{s_1}\gamma_j^{(1)}}{\log s_1} + \frac{\sum_{j=1}^{s_2}\gamma_j^{(2)}}{\log s_2}\right)<\infty.
\end{equation}
We have weak QMC-tractability of multivariate integration in $\cH(K_{\bss,\bsalpha,\bsgamma})$ iff 
\begin{equation}\label{eqcondweaktract}
 \lim_{(s_1 +s_2)\To\infty}\left(\frac{\sum_{j=1}^{s_1}\gamma_j^{(1)}}{s_1} + \frac{\sum_{j=1}^{s_2}\gamma_j^{(2)}}{s_2}\right)=0.
\end{equation}
\end{theorem}

The necessity of the conditions in Theorem~\ref{thm_main} will be proven in Section~\ref{pr_nec} and the sufficiency in  Section~\ref{pr_suf}. 
In the latter section we will see that the notions of tractability can be achieved by using so-called hybrid point sets made 
of polynomial lattice point sets and of classical lattice point sets. We will construct these by a component-by-component algorithm.

\section{Proof of the necessary conditions}\label{pr_nec}

First we prove the following theorem.

\begin{theorem}\label{thmlowerbound}
For any QMC rule using a point set $\cS_{N,\bss}=((\bsx_n,\bsy_n))_{n=0}^{N-1}\subseteq [0,1)^{s}$, we have
\begin{equation}\label{eqthmlowerbound}
e^2(\cH(K_{\bss,\bsalpha,\bsgamma}),\cS_{N,\bss})\ge -1 +
\frac{1}{N}\left(\prod_{j=1}^{s_1}(1+\gamma_j^{(1)} \mu (\alpha_1))\right)\left(\prod_{j=1}^{s_2}(1+2\gamma_j^{(2)}\zeta(\alpha_2))\right),
\end{equation}
where $\mu (\alpha):=\frac{b^{\alpha} (b-1)}{b^{\alpha}-b}$ for $\alpha>1$, and where $\zeta$ denotes the Riemann zeta function. 
\end{theorem}

\begin{proof}
Let us, for the sake of simplicity, 
assume that $1\ge \gamma_j^{(1)}$, respectively, $1\ge \gamma_j^{(2)}$, for $j\ge 1$. From~\eqref{eqkernelproduct}, we
have
\[K_{\bss,\bsalpha,\bsgamma}((\bsx,\bsy),(\bsx',\bsy'))=K_{s_1,\alpha_1,\bsgamma^{(1)}}(\bsx,\bsx') K_{s_2,\alpha_2,\bsgamma^{(2)}}(\bsy,\bsy').\]
Hence, using results in \cite{DP05} and \cite{SW01}, we can write 
\begin{eqnarray*}
\lefteqn{K_{\bss,\bsalpha,\bsgamma}((\bsx,\bsy),(\bsx',\bsy'))
=K_{s_1,\alpha_1,\bsgamma^{(1)}}(\bsx,\bsx') K_{s_2,\alpha_2,\bsgamma^{(2)}}(\bsy,\bsy')}\\
&=& \left(\prod_{j=1}^{s_1} (1+ \gamma_j \phi_{\wal,\alpha_1}(x_j,x_j'))\right)\left(1+2\gamma_j \sum_{l=1}^\infty 
    \frac{\cos (2\pi l (y_j-y_j'))}{l^{\alpha_2}}\right), 
\end{eqnarray*}
where the technical function $\phi_{\wal,\alpha_1}$ is defined as in \cite[p.~170]{DP05}, where it is also noted that 
$1+ \gamma_j \phi_{\wal,\alpha_1}(u,v) \ge 0$ for any $u,v$. Hence we can conclude that $K_{\bss,\bsalpha,\bsgamma}((\bsx,\bsy),(\bsx',\bsy'))$ is nonnegative. Now we use \eqref{wcerror} and obtain 
\begin{eqnarray*}
e^2(\cH(K_{\bss,\bsalpha,\bsgamma}),\cS_{N,\bss})&=&-1+\frac{1}{N^2} \sum_{n,n'=0}^{N-1} 
K_{\bss,\bsalpha,\bsgamma}((\bsx_n,\bsy_n),(\bsx_{n'},\bsy_{n'}))\\
&\ge& -1 +\frac{1}{N^2} \sum_{n=0}^{N-1} 
K_{\bss,\bsalpha,\bsgamma}((\bsx_n,\bsy_n),(\bsx_{n},\bsy_{n}))\\
&=& -1 +\frac{1}{N^2} \sum_{n=0}^{N-1} \left( \sum_{\bsk\in\Natural_0^{s_1}}\sum_{\bsl\in\Integer^{s_2}} r_{\alpha_1,\bsgamma^{(1)}}^{(1)} (\bsk) 
r_{\alpha_2,\bsgamma^{(2)}}^{(2)} (\bsl)\right)\\
&=&-1+\frac{1}{N} \left(\sum_{\bsk\in\Natural_0^{s_1}}r_{\alpha_1,\bsgamma^{(1)}}^{(1)} (\bsk)\right)
\left(\sum_{\bsl\in\Integer^{s_2}}r_{\alpha_2,\bsgamma^{(2)}}^{(2)} (\bsl)\right)\\
&=& -1 +\frac{1}{N}\left(\prod_{j=1}^{s_1}(1+\gamma_j^{(1)} \mu (\alpha_1))\right)\left(\prod_{j=1}^{s_2}(1+2\gamma_j^{(2)}\zeta(\alpha_2))\right).
\end{eqnarray*}
\end{proof}

From Theorem \ref{thmlowerbound}, we immediately obtain that
\[N_{\min}(\varepsilon,s_1+s_2)\ge \frac{1}{1+\varepsilon^2} 
  \left(\prod_{j=1}^{s_1}(1+\gamma_j^{(1)} \mu (\alpha_1))\right)\left(\prod_{j=1}^{s_2}(1+2\gamma_j^{(2)}\zeta(\alpha_2))\right).\]
Now the two products can be analyzed in the same way as it was done in \cite{DP05} and \cite{SW01}, respectively. 
This finally leads to the necessary conditions \eqref{eqcondstrongtract} and \eqref{eqcondtract} in Theorem~\ref{thm_main}. 
Now assume that we have weak QMC-tractability. Then for $\varepsilon=1$ we have 
$$\log N_{\min}(1,s_1+s_2) \ge \log \frac{1}{2}+ \sum_{j=1}^{s_1} \log(1+\gamma_j^{(1)} \mu(\alpha_1))+\sum_{j=1}^{s_2}\log (1+2\gamma_j^{(2)}\zeta(\alpha_2))$$ 
and 
$$\lim_{(s_1+s_2) \rightarrow \infty} \frac{\sum_{j=1}^{s_1} \log(1+\gamma_j^{(1)} \mu(\alpha_1))+\sum_{j=1}^{s_2}
\log (1+2\gamma_j^{(2)}\zeta(\alpha_2))}{s_1+s_2}=0.$$ 
This implies that $\lim_{j \rightarrow \infty} \gamma_j^{(k)}=0$ for $k \in \{1,2\}$. 
For small enough $x >0$ we have $\log(1+x) \ge cx$ for some $c>0$. Hence, 
for some $j_1,j_2 \in \Natural$ and $s_1 \ge j_1$ and $s_2 \ge j_2$ we have 
\begin{eqnarray*}
\lefteqn{\sum_{j=1}^{s_1} \log(1+\gamma_j^{(1)} \mu(\alpha_1))+\sum_{j=1}^{s_2}\log (1+2\gamma_j^{(2)}\zeta(\alpha_2))}\\
&& \ge c_1 \mu(\alpha_1) \sum_{j = j_1}^{s_1} \gamma_j^{(1)} + c_2 2 \zeta(\alpha_2) \sum_{j = j_2}^{s_2} \gamma_j^{(2)}
\end{eqnarray*}
and therefore, under the assumption of weak QMC-tractability, 
$$\lim_{(s_1+s_2)\rightarrow \infty}\frac{c_1 \mu(\alpha_1) \sum_{j = j_1}^{s_1} \gamma_j^{(1)} + c_2 2 \zeta(\alpha_2) 
\sum_{j = j_2}^{s_2} \gamma_j^{(2)}}{s_1+s_2}=0.$$ 
This implies the necessity of \eqref{eqcondweaktract}.

\section{Proof of the sufficient conditions}\label{pr_suf}

We give a constructive proof of the sufficient conditions by finding, component-by-component, 
a QMC algorithm whose worst-case error implies the sufficient conditions in 
Theorem~\ref{thm_main}. This QMC algorithm is based on lattice rules and on polynomial lattice rules, 
where the lattice rules are used to integrate the ``Korobov part'' of the integrand and the polynomial lattce rules are used 
to integrate the ``Walsh part''. We quickly recall the concepts of (polynomial) lattice rules:
\begin{itemize}
\item {\bf Lattice point sets (according to Hlawka \cite{H62} and Korobov \cite{K59}).}
Let $N \in \Natural$ be an integer and let $\bsz=(z_1,\ldots,z_{s_2}) \in
\ZZ^{s_2}$. The {\it lattice point set} $(\bsy_n)_{n=0}^{N-1}$
with {\it generating vector} $\bsz$, consisting of $N$ points in
$[0,1)^{s_2}$, is defined by
\[\bsy_n=\left(\left\{\frac{nz_1}{N}\right\},\ldots,\left\{\frac{nz_{s_2}}{N}\right\}\right)
 \; \mbox{ for all }\; 0\le n\le N-1,\] where $\{\cdot\}$ denotes the fractional part
of a number. A QMC rule that is based on a lattice point set is called a {\it lattice rule}.

\item {\bf Polynomial lattice point sets (according to Niederreiter \cite{nie92}).} Let $\FF_b$ be the finite field of prime order $b$.
Furthermore let $\FF_b[x]$ be the set of polynomials over
$\FF_b$, and let $\FF_b ((x^{-1}))$ be the field of formal Laurent
series over $\FF_b$. The latter contains the field of rational functions
as a subfield. Given $m\in \Natural$, set $G_{b,m}:=\{a \in \FF_b[x]\, : \, \deg(a)<m\}$ and define a mapping
$\nu_m:\FF_b ((x^{-1}))\To [0,1)$ by
\[\nu_m\left(\sum_{l=z}^{\infty}t_l x^{-l}\right):=\sum_{l=\max (1,z)}^{m}t_l b^{-l}.\]
Let $f\in\FF_b [x]$ with $\deg(f)=m$ and $\bsg=(g_1,\ldots,g_{s_1})\in\FF_b [x]^{s_1}$. The {\it polynomial lattice point set} $(\bsx_h)_{h \in G_{b,m}}$ with generating vector $\bsg$, consisting of $b^m$ points in $[0,1)^{s_1}$, is defined by   
\[\bsx_h:=\left(\nu_m\left(\frac{h(x)g_1 (x)}{f(x)}\right),\ldots,\nu_m\left(\frac{h(x)g_{s_1} (x)}{f(x)}\right)\right) \; \mbox{ for all }\; h \in G_{b,m}.\]
A QMC rule using a polynomial lattice point set is called {\it polynomial lattice rule}. 
\end{itemize}

\subsection{Component-by-component construction}

We now show a component-by-component (CBC) construction algorithm for point sets that are suitable for integration in 
the space $\cH(K_{\bss,\bsalpha,\bsgamma})$. For practical reasons, we will, in the following, denote the worst-case 
error of a hybrid point set $\cS_{N,\bss}=((\bsx_n,\bsy_n))_{n=0}^{N-1}$, consisting of an $s_1$-dimensional polynomial lattice 
generated by $\bsg$ and an $s_2$-dimensional 
lattice generated by $\bsz$, by 
$$e^2_{\bss,\bsalpha,\bsgamma}(\bsg,\bsz),$$
where $\bsg$ is the generating vector of the polynomial lattice part, and $\bsz$ is the generating vector of the lattice part. We have
\begin{eqnarray}\label{eqerrexpression}
e^2_{\bss,\bsalpha,\bsgamma}(\bsg,\bsz) & = & -1 + 
\frac{1}{N^2} \sum_{n,n'=0}^{N-1} \left[\prod_{j=1}^{s_1}\left(1+\gamma_j^{(1)} \sum_{k_j\in\Natural} 
\frac{\wal_{k_j}(x_n^{(j)}\ominus x_{n'}^{(j)})}{b^{\alpha_1\psi_b (k_j)}}\right)\right]\nonumber \\
&& \hspace{1cm} \times 
\left[\prod_{j=1}^{s_2}\left(1+\gamma_j^{(2)} \sum_{l_j\in\Integer\setminus\{0\}} 
\frac{\ee_{l_j}(y_n^{(j)}-y_{n'}^{(j)})}{\abs{l_j}^{\alpha_2}}\right)\right].
\end{eqnarray}

We now proceed to our construction algorithm. Note that we state the algorithm in a way such that
we exclude the cases $s_1=0$ or $s_2=0$, as these are covered by the results in \cite{DKPS05} and \cite{KJ02}.
\begin{algorithm}\label{algcbchybrid}
Let $s_1,s_2,m\in \Natural$, a prime number $b$, and an irreducible polynomial $f\in \FF_b[x]$ with $\deg (f)=m$ be given. 
We write $N=b^m$. 
\begin{enumerate}
\item For $d_1=1$, choose $g_1=1\in G_{b,m}$.
\item For $d_2=1$, choose $z_1\in Z_N$ such that
$$e^2_{(1,1),\bsalpha,\bsgamma}(g_1,z_1)$$
is minimized as a function of $z_1$.
\item For $d_1\in\{1,\ldots,s_1\}$ and $d_2\in\{1,\ldots,s_2\}$, assume that 
$\bsg_{d_1}^*=(g_1,\ldots,g_{d_1})$ and $\bsz_{d_2}^*=(z_1,\ldots,z_{d_2})$ are given. 
If $d_1< s_1$ and $d_2<s_2$ go to either Step (\ref{stepa}) or (\ref{stepb}). If $d_1=s_1$ and $d_2<s_2$ 
go to Step (\ref{stepb}). If $d_1<s_1$ and $d_2=s_2$, go to Step (\ref{stepa}). If $d_1=s_1$ and $d_2=s_2$, the 
algorithm terminates. 
\begin{enumerate}
\item \label{stepa} Choose $g_{d_1+1}\in G_{b,m}$ such that
 $$e^2_{(d_1+1,d_2),\bsalpha,\bsgamma}((\bsg_{d_1}^*,g_{d_1+1}),\bsz_{d_2}^*)$$
is minimized as a function of $g_{d_1+1}$. Increase $d_1$ by 1 and repeat Step 3.
\item \label{stepb} Choose $g_{d_2+1}\in Z_N$ such that 
 $$e^2_{(d_1,d_2+1),\bsalpha,\bsgamma}(\bsg_{d_1}^*,(\bsz_{d_2}^*,z_{d_2+1}))$$
is minimized as a function of $z_{d_2+1}$. Increase $d_2$ by 1 and repeat Step 3.
\end{enumerate}
\end{enumerate}
\end{algorithm}
\begin{remark}
 As pointed out in, e.g., \cite{SW01} and \cite{DP05}, the infinite sums in~\eqref{eqerrexpression} can be represented in closed form, 
 so the construction cost of Algorithm~\ref{algcbchybrid} is of order $\mathcal{O}(n^3 (s_1 +s_2)^2)$.
\end{remark}

The following theorem states that our algorithm yields hybrid integration nodes with 
a low worst case error.

\begin{theorem}\label{thmcbc} Let $d_1\in\{1,\ldots,s_1\}$ and $d_2\in\{1,\ldots,s_2\}$ be given. Then 
the generating vectors $\bsg_{d_1}^*$ and $\bsz_{d_2}^*$ constructed by Algorithm \ref{algcbchybrid} satisfy
\begin{equation}\label{eqthmcbc}
e^2_{(d_1,d_2),\bsalpha,\bsgamma}(\bsg_{d_1}^*,\bsz_{d_2}^*)\le \frac{2}{N} 
\left(\prod_{j=1}^{d_1}\left(1+\gamma_j^{(1)}2\mu (\alpha_1)\right)\right) 
\left(\prod_{j=1}^{d_2}\left(1+\gamma_j^{(2)}4\zeta (\alpha_2)\right)\right).
\end{equation}
\end{theorem}

The proof of Theorem~\ref{thmcbc} is deferred to the appendix.

\subsection{Proof of the sufficient conditions}

Let us first deal with strong tractability. From Theorem \ref{thmcbc} it follows that for $N=b^m$ the squared $N$th minimal error satisfies
\begin{eqnarray*}
e^2(N,s_1+s_2) \le  \frac{2}{N} 
\left(\prod_{j=1}^{s_1}\left(1+\gamma_j^{(1)}2\mu (\alpha_1)\right)\right) 
\left(\prod_{j=1}^{s_2}\left(1+\gamma_j^{(2)}4\zeta (\alpha_2)\right)\right).
\end{eqnarray*}
Assuming that \eqref{eqcondstrongtract} holds, we know that $\sum_{j=1}^{\infty}\gamma_j^{(1)} \le \infty$,
and hence we can estimate
\begin{align*}
\prod_{j=1}^{s_1}(1+\gamma_j^{(1)} \mu (\alpha_1))&=\exp\left(\sum_{j=1}^{s_1} \log (1+\gamma_j^{(1)} \mu (\alpha_1))\right)\\
& \le  \exp\left(\sum_{j=1}^{\infty} \gamma_j^{(1)} \mu (\alpha_1)\right)=:C_1(\alpha_1,\bsgamma^{(1)}).
\end{align*}
A similar argument shows that
\[\prod_{j=1}^{s_2}(1+4\gamma_j^{(2)} \zeta (\alpha_2)) \le C_2(\alpha_2,\bsgamma^{(2)}).\]
Hence 
$$e^2(N,s_1+s_2) \le \frac{2}{N} C_1(\alpha_1,\bsgamma^{(1)}) C_2(\alpha_2,\bsgamma^{(2)})=: \frac{C(\bsalpha,\bsgamma)}{N}.$$
For $\varepsilon>0$ choose $m \in \Natural$ such that $b^{m-1} < \lceil C(\bsalpha,\bsgamma)\varepsilon^{-2} \rceil =:N' \le b^m$. 
Then we have $e(b^m,s_1+s_2) \le \varepsilon$ and hence 
$$N_{\min}(\varepsilon,s_1+s_2) \le b^m < b N' =b \lceil C(\bsalpha,\bsgamma) \varepsilon^{-2} \rceil .$$ 
This implies strong polynomial QMC-tractability. The corresponding bounds can be achieved with the point set constructed by Algorithm~\ref{algcbchybrid}.

The sufficiency of the condition for polynomial QMC-tractability is shown in a similar fashion by standard arguments (cf. \cite{DP05,SW01}). 

We show the sufficiency of the condition for weak QMC-tractability. From Theorem~\ref{thmcbc} it follows that 
$$N_{\min}(\varepsilon,s_1+s_2) \le 2 \varepsilon^{-2} \left\lceil\left(\prod_{j=1}^{s_1}\left(1+\gamma_j^{(1)}2\mu (\alpha_1)\right)\right) 
\left(\prod_{j=1}^{s_2}\left(1+\gamma_j^{(2)}4\zeta (\alpha_2)\right)\right)\right\rceil.$$ 
Hence
\begin{eqnarray*}
\log N_{\min}(\varepsilon,s_1+s_2) & \le & \log 4 +2 \log \varepsilon^{-1}\\
&& +\sum_{j=1}^{s_1} \log \left(1+\gamma_j^{(1)}2\mu (\alpha_1)\right)+ \sum_{j=1}^{s_2}\log\left(1+\gamma_j^{(2)}4\zeta (\alpha_2)\right)\\
&\le& \log 4 +2 \log \varepsilon^{-1} +2\mu (\alpha_1) \sum_{j=1}^{s_1} \gamma_j^{(1)}+ 4\zeta (\alpha_2) \sum_{j=1}^{s_2} \gamma_j^{(2)},
\end{eqnarray*}
and this together with \eqref{eqcondweaktract} implies the result.
 
\section*{Appendix: The proof of Theorem~\ref{thmcbc}}\label{proof_thm3}

We now give the proof of Theorem~\ref{thmcbc}.

\begin{proof}
We show the result by an inductive argument. 

Note that we excluded the cases where $s_1=0$ or where $s_2=0$, so we start our considerations by dealing with the case 
where $d_1=d_2=1$. According to Algorithm \ref{algcbchybrid}, we have chosen $g_1=1\in G_{b,m}$ and $z_1\in Z_N$ such that
$$e^2_{(1,1),\bsalpha,\bsgamma}(g_1,z_1)$$ is minimized as a function of $z_1$. In the following, we denote the 
points generated by $(g,z)\in G_{b,m}\times Z_N$ by $(x_n (g), y_n (z))$.

According to Equation \eqref{eqerrexpression}, we have
$$e^2_{(1,1),\bsalpha,\bsgamma}(g_1,z_1)= e^2_{1,\alpha_1,\gamma^{(1)}}(1)+\theta_{(1,1)}(z_1),$$
where $e^2_{1,\alpha_1,\gamma^{(1)}}(1)$ denotes the squared worst-case error of the polynomial lattice rule generated by 1 in the Walsh space 
$\cH(K_{1,\alpha_1,\gamma^{(1)}})$, and where
\begin{multline*}
\theta_{(1,1)}(z_1):=\frac{\gamma_{1}^{(2)}}{N^2} \sum_{n,n'=0}^{N-1} \left(1+\gamma_1^{(1)}\sum_{k_1\in\Natural}\frac{\wal_{k_1}(x_n^{(1)}(g_1)
\ominus x_{n'}^{(1)}(g_1))}{b^{\alpha_1\psi_b (k_1)}}\right)\\ \times\sum_{l_1\in\Integer\setminus\{0\}}
\frac{\ee_{l_1}(y_n^{(1)}(z_1)-y_{n'}^{(1)}(z_1))}{\abs{l_1}^{\alpha_2}}.
\end{multline*}
By results in \cite{DKPS05}, we know that 
\begin{equation}\label{eqonedimDKPS}
e^2_{1,\alpha_1,\gamma^{(1)}}(1) \le \frac{2}{N}\left(1+\gamma_1^{(1)}\mu(\alpha_1)\right). 
\end{equation}

Then, as $z_1$ was chosen to minimize the error,
\begin{eqnarray*}
\theta_{(1,1)}(z_1)&\le&\frac{1}{\phi (N)} \sum_{z\in Z_N} \theta_{(1,1)}(z)\\
&=&\frac{\gamma_{1}^{(2)}}{N^2} \sum_{n,n'=0}^{N-1} \left(1+\gamma_1^{(1)}\sum_{k_1\in\Natural}\frac{\wal_{k_1}(x_n^{(1)}(g_1)
\ominus x_{n'}^{(1)}(g_1))}{b^{\alpha_1\psi_b (k_1)}}\right)\\ &&\times\frac{1}{\phi (N)} \sum_{z\in Z_N} \sum_{l_1\in\Integer\setminus\{0\}}
\frac{\ee_{l_1}(y_n^{(1)}(z)-y_{n'}^{(1)}(z))}{\abs{l_1}^{\alpha_2}}\\
&\le&\frac{\gamma_{1}^{(2)}}{N^2} \sum_{n,n'=0}^{N-1} \left(1+\gamma_1^{(1)}\mu (\alpha_1)\right)
\abs{\frac{1}{\phi (N)} \sum_{z\in Z_N} \sum_{l_1\in\Integer\setminus\{0\}}
\frac{\ee_{l_1}(y_n^{(1)}(z)-y_{n'}^{(1)}(z))}{\abs{l_1}^{\alpha_2}}}\\
&=&\gamma_{1}^{(2)} \left(1+\gamma_1^{(1)}\mu (\alpha_1)\right)\frac{1}{N^2}\sum_{n,n'=0}^{N-1}
\abs{\frac{1}{\phi (N)} \sum_{z\in Z_N} \sum_{l_1\in\Integer\setminus\{0\}}
\frac{\ee_{l_1}(y_n^{(1)}(z)-y_{n'}^{(1)}(z))}{\abs{l_1}^{\alpha_2}}}.
\end{eqnarray*}
Let us now deal with the term 
\begin{eqnarray*}
\Sigma_B&:=&\frac{1}{N^2}\sum_{n,n'=0}^{N-1}
\abs{\frac{1}{\phi (N)} \sum_{z\in Z_N} \sum_{l_1\in\Integer\setminus\{0\}}
\frac{\ee_{l_1}(y_n^{(1)}(z)-y_{n'}^{(1)}(z))}{\abs{l_1}^{\alpha_2}}}\\
&=&\frac{1}{N^2}\sum_{n=0}^{N-1}\sum_{n'=0}^{N-1}
\abs{\frac{1}{\phi (N)} \sum_{z\in Z_N} \sum_{l_1\in\Integer\setminus\{0\}}
\frac{\ee^{2\pi\icomp (n-n')zl_1/N}}{\abs{l_1}^{\alpha_2}}}\\
&=&\frac{1}{N}\sum_{n=0}^{N-1}
\abs{\frac{1}{\phi (N)} \sum_{z\in Z_N} \sum_{l_1\in\Integer\setminus\{0\}}
\frac{\ee^{2\pi\icomp n z l_1/N}}{\abs{l_1}^{\alpha_2}}}\\
&=&\frac{1}{N}\sum_{n=1}^{N}
\abs{\frac{1}{\phi (N)} \sum_{z\in Z_N} \sum_{l_1\in\Integer\setminus\{0\}}
\frac{\ee^{2\pi\icomp n z l_1/N}}{\abs{l_1}^{\alpha_2}}},
\end{eqnarray*}
since the inner sum in the second line always has the same value. We now use \cite[Lemmas 2.1 and 2.3]{KJ02} and obtain 
$$\Sigma_B\le\frac{1}{N}4\zeta (\alpha_2),$$
where we used that $N$ has only one prime factor. Hence we obtain 
\begin{equation}\label{eqonedimKJ}
\theta_{(1,1)}(z_1)\le \frac{\gamma_{1}^{(2)}}{N}\left(1+\gamma_1^{(1)}\mu (\alpha_1)\right)4\zeta (\alpha_2).
\end{equation}
Combining Equations \eqref{eqonedimDKPS} and \eqref{eqonedimKJ} yields the desired bound for $(g_1,z_1)$.

\medskip

Let us now assume $d_1\in\{1,\ldots,s_1\}$ and $d_2\in\{1,\ldots,s_2\}$ and that we have already found generating 
vectors $\bsg_{d_1}^*$ and $\bsz_{d_2}^*$ such that Equation \eqref{eqthmcbc} is satisfied. 

In what follows, we are going to distinguish two cases: In the first case, we assume that $d_1<s_1$ and
add a component $g_{d_1+1}$ to $\bsg_{d_1}^*$, and 
in the second case, we assume that $d_2<s_2$ and add a component $z_{d_2+1}$ to $\bsz_{d_2}^*$. In both cases, we will show that the corresponding 
bounds on the squared worst-case errors hold. 

\medskip

Let us first consider the case where we start from $(\bsg_{d_1}^*,\bsz_{d_2}^*)$ and add, by Algorithm \ref{algcbchybrid}, a component $g_{d_1+1}$ to $\bsg_{d_1}^*$.
According to Equation \eqref{eqerrexpression}, we have
\begin{eqnarray*}
e^2_{(d_1+1,d_2),\bsalpha,\bsgamma}((\bsg_{d_1}^*,g_{d_1+1}),\bsz_{d_2}^*)=e^2_{(d_1,d_2),\bsalpha,\bsgamma}(\bsg_{d_1}^*,\bsz_{d_2}^*)+ \theta_{(d_1+1,d_2)}(g_{d_1+1}),
\end{eqnarray*}
where 
\begin{eqnarray*}
\lefteqn{\theta_{(d_1+1,d_2)}(g_{d_1+1})}\\&:=&\frac{\gamma_{d_1 +1}^{(1)}}{N^2} \sum_{n,n'=0}^{N-1} \left[\prod_{j=1}^{d_1}\left(1+\gamma_j^{(1)}\sum_{k_j\in\Natural}\frac{\wal_{k_j}(x_n^{(j)}(g_j)
\ominus x_{n'}^{(j)}(g_j))}{b^{\alpha_1\psi_b (k_j)}}\right)\right]\\
&&\times\left[\prod_{j=1}^{d_2}\left(1+\gamma_j^{(2)}\sum_{l_j\in\Integer\setminus\{0\}}\frac{\ee_{l_j}(y_n^{(j)}(z_j)
- y_{n'}^{(j)}(z_j))}{\abs{l_j}^{\alpha_2}}\right)\right]\\
&&\times\sum_{k_{d_1+1}\in\Natural}\frac{\wal_{k_{d_1+1}}(x_n^{(d_1+1)}(g_{d_1+1})
\ominus x_{n'}^{(d_1+1)}(g_{d_1+1}))}{b^{\alpha_1\psi_b (k_{d_1+1})}}.
\end{eqnarray*}
However, by the assumption, we know that 
\begin{eqnarray}\label{eqindassumption}
\lefteqn{e^2_{(d_1,d_2),\bsalpha,\bsgamma}(\bsg_{d_1}^*,\bsz_{d_2}^*)}\nonumber \\ &\le& \frac{2}{N} 
\left[\prod_{j=1}^{d_1}\left(1+\gamma_j^{(1)}2\mu (\alpha_1)\right)\right] 
\left[\prod_{j=1}^{d_2}\left(1+\gamma_j^{(2)}4\zeta (\alpha_2)\right)\right].
\end{eqnarray}

Furthermore, as $g_{d_1+1}$ was chosen to minimize the error,
\begin{eqnarray*}
\lefteqn{\theta_{(d_1+1,d_2)}(g_{d_1+1})\le\frac{1}{N} \sum_{g\in G_{b,m}} \theta_{(d_1+1,d_2)}(g)}\\
&=&\frac{\gamma_{d_1 +1}^{(1)}}{N^2} \sum_{n,n'=0}^{N-1} \left[\prod_{j=1}^{d_1}\left(1+\gamma_j^{(1)}\sum_{k_j\in\Natural}\frac{\wal_{k_j}(x_n^{(j)}(g_j)
\ominus x_{n'}^{(j)}(g_j))}{b^{\alpha_1\psi_b (k_j)}}\right)\right]\\
&&\times\left[\prod_{j=1}^{d_2}\left(1+\gamma_j^{(2)}\sum_{l_j\in\Integer\setminus\{0\}}\frac{\ee_{l_j}(y_n^{(j)}(z_j)
- y_{n'}^{(j)}(z_j))}{\abs{l_j}^{\alpha_2}}\right)\right]\\
&&\times\frac{1}{N} \sum_{g\in G_{b,m}}\sum_{k_{d_1+1}\in\Natural}\frac{\wal_{k_{d_1+1}}(x_n^{(d_1+1)}(g)
\ominus x_{n'}^{(d_1+1)}(g))}{b^{\alpha_1\psi_b (k_{d_1+1})}}\\
&\le& \frac{\gamma_{d_1 +1}^{(1)}}{N^2} \sum_{n,n'=0}^{N-1} \left[\prod_{j=1}^{d_1}\left(1+\gamma_j^{(1)}\mu(\alpha_1)\right)\right]
\left[\prod_{j=1}^{d_2}\left(1+\gamma_j^{(2)}2\zeta(\alpha_2)\right)\right]\\
&&\times\abs{\frac{1}{N} \sum_{g\in G_{b,m}}\sum_{k_{d_1+1}\in\Natural}\frac{\wal_{k_{d_1+1}}(x_n^{(d_1+1)}(g)
\ominus x_{n'}^{(d_1+1)}(g))}{b^{\alpha_1\psi_b (k_{d_1+1})}}}\\
&=&\gamma_{d_1 +1}^{(1)}\left[\prod_{j=1}^{d_1}\left(1+\gamma_j^{(1)}\mu(\alpha_1)\right)\right]
\left[\prod_{j=1}^{d_2}\left(1+\gamma_j^{(2)}2\zeta(\alpha_2)\right)\right]\\
&&\times\frac{1}{N^2}\sum_{n,n'=0}^{N-1}\abs{\frac{1}{N} \sum_{g\in G_{b,m}}\sum_{k_{d_1+1}\in\Natural}\frac{\wal_{k_{d_1+1}}(x_n^{(d_1+1)}(g)
\ominus x_{n'}^{(d_1+1)}(g))}{b^{\alpha_1\psi_b (k_{d_1+1})}}}.
\end{eqnarray*}
Let us now deal with the term 
\begin{eqnarray*}
\Sigma_C&:=&\frac{1}{N^2}\sum_{n,n'=0}^{N-1}\abs{\frac{1}{N} \sum_{g\in G_{b,m}}\sum_{k_{d_1+1}\in\Natural}\frac{\wal_{k_{d_1+1}}(x_n^{(d_1+1)}(g)
\ominus x_{n'}^{(d_1+1)}(g))}{b^{\alpha_1\psi_b (k_{d_1+1})}}}\\
&=&\frac{1}{N^2}\sum_{n=0}^{N-1}\sum_{n'=0}^{N-1}\abs{\frac{1}{N} \sum_{g\in G_{b,m}}\sum_{k_{d_1+1}\in\Natural}\frac{\wal_{k_{d_1+1}}(x_{n\ominus n'}^{(d_1+1)}(g))}
{b^{\alpha_1\psi_b (k_{d_1+1})}}}\\
&=&\frac{1}{N}\sum_{n=0}^{N-1}\abs{\frac{1}{N} \sum_{g\in G_{b,m}}\sum_{k_{d_1+1}\in\Natural}
\frac{\wal_{k_{d_1+1}}(x_n^{(d_1+1)}(g))}{b^{\alpha_1\psi_b (k_{d_1+1})}}},
\end{eqnarray*}
where we again used that the inner sum in the second line always has the same value. We now write
\begin{eqnarray*}
 \Sigma_C&=&\frac{1}{N}\sum_{k_{d_1+1}\in\Natural}\frac{1}{b^{\alpha_1\psi_b (k_{d_1+1})}}\\
&& +\frac{1}{N}\sum_{n=1}^{N-1}\abs{\frac{1}{N} \sum_{g\in G_{b,m}}\sum_{k_{d_1+1}\in\Natural}
\frac{\wal_{k_{d_1+1}}(x_n^{(d_1+1)}(g))}{b^{\alpha_1\psi_b (k_{d_1+1})}}}\\
&=&\frac{\mu(\alpha_1)}{N}
+\frac{1}{N}\sum_{n=1}^{N-1}\abs{\frac{1}{N} \sum_{g\in G_{b,m}}\sum_{k_{d_1+1}\in\Natural}
\frac{\wal_{k_{d_1+1}}(x_n^{(d_1+1)}(g))}{b^{\alpha_1\psi_b (k_{d_1+1})}}}.
\end{eqnarray*}
Let now $n\in\{1,\ldots,N-1\}$ be fixed, and consider the term
$$\Sigma_{C,n}:=\frac{1}{N} \sum_{g\in G_{b,m}}\sum_{k_{d_1+1}\in\Natural}
\frac{\wal_{k_{d_1+1}}(x_n^{(d_1+1)}(g))}{b^{\alpha_1\psi_b (k_{d_1+1})}}.$$

We obtain
\begin{eqnarray*}
 \Sigma_{C,n}&=&\sum_{k_{d_1+1}\in\Natural} \frac{1}{N}\sum_{g\in G_{b,m}}
\frac{\wal_{k_{d_1+1}}(x_n^{(d_1+1)}(g))}{b^{\alpha_1\psi_b (k_{d_1+1})}}\\
&=&\sum_{\substack{k_{d_1+1}\in\Natural\\ k_{d_1+1}\equiv 0 (N)}}\frac{1}{N} \sum_{g\in G_{b,m}}
\frac{\wal_{k_{d_1+1}}(x_n^{(d_1+1)}(g))}{b^{\alpha_1\psi_b (k_{d_1+1})}}\\
&&+\sum_{\substack{k_{d_1+1}\in\Natural\\ k_{d_1+1}\not\equiv 0 (N)}}\frac{1}{N} \sum_{g\in G_{b,m}}
\frac{\wal_{k_{d_1+1}}(x_n^{(d_1+1)}(g))}{b^{\alpha_1\psi_b (k_{d_1+1})}}\\
&=:&\Sigma_{C,n,1}+\Sigma_{C,n,2}.
\end{eqnarray*}
By results in \cite{DKPS05}, 
$$\Sigma_{C,n,1}=\sum_{\substack{k_{d_1+1}\in\Natural\\ k_{d_1+1}\equiv 0 (N)}}\frac{1}{b^{\alpha_1\psi_b (k_{d_1+1})}}=
\frac{\mu(\alpha_1)}{b^{m\alpha}}\le\frac{\mu(\alpha_1)}{N}.$$
Furthermore,
\begin{eqnarray*}
 \Sigma_{C,n,2}&=&\sum_{\substack{k_{d_1+1}\in\Natural\\ k_{d_1+1}\not\equiv 0 (N)}}
\frac{1}{b^{\alpha_1\psi_b (k_{d_1+1})}}\frac{1}{N} \sum_{g\in G_{b,m}}
\wal_{k_{d_1+1}}(x_n^{(d_1+1)}(g)).
\end{eqnarray*}
Note that 
$$\sum_{g\in G_{b,m}}\wal_{k_{d_1+1}}(x_n^{(d_1+1)}(g))=\sum_{g\in G_{b,m}}\wal_{k_{d_1+1}}
\left(\nu_m \left(\frac{n(x)g(x)}{f(x)}\right)\right).$$
Since $n\neq 0$, we can write 
$$\Sigma_{C,n,2}=\sum_{\substack{k_{d_1+1}\in\Natural\\ k_{d_1+1}\not\equiv 0 (N)}}\frac{1}{b^{\alpha_1\psi_b (k_{d_1+1})}}\frac{1}{N}
\sum_{g\in G_{b,m}}\wal_{k_{d_1+1}} \left(\nu_m \left(\frac{g(x)}{f(x)}\right)\right).$$
Since $g$ takes on all values in $G_{b,m}$, and since $f$ is irreducible, we can simplify $\Sigma_{C,n,2}$ further to 
$$\Sigma_{C,n,2}=\sum_{\substack{k_{d_1+1}\in\Natural\\ k_{d_1+1}\not\equiv 0 (N)}}\frac{1}{b^{\alpha_1\psi_b (k_{d_1+1})}}\frac{1}{N}
 \sum_{g=0}^{b^m-1}\wal_{k_{d_1+1}}\left(\frac{g}{b^m}\right).$$
However, $ \sum_{g=0}^{b^m-1}\wal_{k_{d_1+1}}\left(\frac{g}{b^m}\right)=0$ and so $\Sigma_{C,n,2}=0$. This yields 
$$\abs{\Sigma_{C,n}}\le\frac{\mu(\alpha_1)}{N},\ \mbox{ and } \ \Sigma_C \le \frac{2\mu(\alpha_1)}{N}.$$
This implies 
\begin{align*}
\theta_{(d_1+1,d_2)}(g_{d_1+1})\le &  \frac{2\gamma_{d_1 +1}^{(1)}\mu(\alpha_1)}{N} \\
& \times \left[\prod_{j=1}^{d_1}\left(1+\gamma_j^{(1)}\mu(\alpha_1)\right)\right]
\left[\prod_{j=1}^{d_2}\left(1+\gamma_j^{(2)}2\zeta(\alpha_2)\right)\right].
\end{align*}
Combining the latter result with Equation \eqref{eqindassumption}, 
we obtain
\begin{eqnarray*}
\lefteqn{e^2_{(d_1+1,d_2),\bsalpha,\bsgamma}((\bsg_{d_1}^*,g_{d_1+1}),\bsz_{d_2}^*))}\\ &\le&
\frac{2}{N} 
\left[\prod_{j=1}^{d_1+1}\left(1+2\gamma_j^{(1)}\mu (\alpha_1)\right)\right] 
\left[\prod_{j=1}^{d_2}\left(1+\gamma_j^{(2)}4\zeta (\alpha_2)\right)\right].
\end{eqnarray*}

Let us now consider the case where we start from 
$(\bsg_{d_1}^*,\bsz_{d_2}^*)$ and add, by Algorithm \ref{algcbchybrid}, a component $z_{d_2+1}$ to $\bsz_{d_2}^*$.
According to Equation \eqref{eqerrexpression}, we have
\begin{eqnarray*}
e^2_{(d_1,d_2+1),\bsalpha,\bsgamma}(\bsg_{d_1}^*,(\bsz_{d_2}^*,z_{d_2+1}))=
e^2_{(d_1,d_2),\bsalpha,\bsgamma}(\bsg_{d_1}^*,\bsz_{d_2}^*)+\theta_{(d_1,d_2+1)}(z_{d_2+1}),
\end{eqnarray*}
where
\begin{eqnarray*}
\lefteqn{\theta_{(d_1,d_2+1)}(z_{d_2+1})}\\
&:=&\frac{\gamma_{d_2 +1}^{(2)}}{N^2} \sum_{n,n'=0}^{N-1} \left[\prod_{j=1}^{d_1}\left(1+\gamma_j^{(1)}\sum_{k_j\in\Natural}\frac{\wal_{k_j}(x_n^{(j)}(g_j)
\ominus x_{n'}^{(j)}(g_j))}{b^{\alpha_1\psi_b (k_j)}}\right)\right]\\
&&\times\left[\prod_{j=1}^{d_2}\left(1+\gamma_j^{(2)}\sum_{l_j\in\Integer\setminus\{0\}}\frac{\ee_{l_j}(y_n^{(j)}(z_j)
- y_{n'}^{(j)}(z_j))}{\abs{l_j}^{\alpha_2}}\right)\right]\\
&&\times\sum_{l_{d_2+1}\in\Integer\setminus\{0\}}\frac{\ee_{l_{d_2+1}}(y_n^{(d_2+1)}(z_{d_2+1})
- y_{n'}^{(d_2+1)}(z_{d_2+1}))}{\abs{l_{d_2+1}}^{\alpha_2}}.
\end{eqnarray*}
By the assumption, we know that 
\begin{equation}\label{eqindassumption2}
e^2_{(d_1,d_2),\bsalpha,\bsgamma}(\bsg_{d_1}^*,\bsz_{d_2}^*)\le \frac{2}{N} 
\left[\prod_{j=1}^{d_1}\left(1+\gamma_j^{(1)}2\mu (\alpha_1)\right)\right] 
\left[\prod_{j=1}^{d_2}\left(1+\gamma_j^{(2)}4\zeta (\alpha_2)\right)\right].
\end{equation}

Then, as $z_{d_2+1}$ was chosen to minimize the error,
\begin{eqnarray*}
\lefteqn{\theta_{(d_1,d_2+1)}(z_{d_2+1})\le\frac{1}{\phi(N)} \sum_{z\in Z_N} \theta_{(d_1,d_2+1)}(z)}\\
&=&\frac{\gamma_{d_2 +1}^{(2)}}{N^2} \sum_{n,n'=0}^{N-1} \left[\prod_{j=1}^{d_1}\left(1+\gamma_j^{(1)}\sum_{k_j\in\Natural}\frac{\wal_{k_j}(x_n^{(j)}(g_j)
\ominus x_{n'}^{(j)}(g_j))}{b^{\alpha_1\psi_b (k_j)}}\right)\right]\\
&&\times\left[\prod_{j=1}^{d_2}\left(1+\gamma_j^{(2)}\sum_{l_j\in\Integer\setminus\{0\}}\frac{\ee_{l_j}(y_n^{(j)}(z_j)
- y_{n'}^{(j)}(z_j))}{\abs{l_j}^{\alpha_2}}\right)\right]\\
&&\times\frac{1}{\phi(N)} \sum_{z\in Z_N}\sum_{l_{d_2+1}\in\Integer\setminus\{0\}}\frac{\ee_{l_{d_2+1}}(y_n^{(d_2+1)}(z)
- y_{n'}^{(d_2+1)}(z))}{\abs{l_{d_2+1}}^{\alpha_2}}\\
&\le& \frac{\gamma_{d_2 +1}^{(2)}}{N^2} \sum_{n,n'=0}^{N-1} \left[\prod_{j=1}^{d_1}\left(1+\gamma_j^{(1)}\mu(\alpha_1)\right)\right]
\left[\prod_{j=1}^{d_2}\left(1+\gamma_j^{(2)}2\zeta(\alpha_2)\right)\right]\\
&&\times\abs{\frac{1}{\phi(N)} \sum_{z\in Z_N}\sum_{l_{d_2+1}\in\Integer\setminus\{0\}}\frac{\ee_{l_{d_2+1}}(y_n^{(d_2+1)}(z)
- y_{n'}^{(d_2+1)}(z))}{\abs{l_{d_2+1}}^{\alpha_2}}}\\
&=&\gamma_{d_2 +2}^{(1)}\left[\prod_{j=1}^{d_1}\left(1+\gamma_j^{(1)}\mu(\alpha_1)\right)\right]
\left[\prod_{j=1}^{d_2}\left(1+\gamma_j^{(2)}2\zeta(\alpha_2)\right)\right]\\
&&\times\frac{1}{N^2}\sum_{n,n'=0}^{N-1}\abs{\frac{1}{\phi(N)} \sum_{z\in Z_N}\sum_{l_{d_2+1}\in\Integer\setminus\{0\}}\frac{\ee_{l_{d_2+1}}(y_n^{(d_2+1)}(z)
- y_{n'}^{(d_2+1)}(z))}{\abs{l_{d_2+1}}^{\alpha_2}}}
\end{eqnarray*}
Let us now deal with the term 
\begin{eqnarray*}
\Sigma_D&:=&\frac{1}{N^2}\sum_{n,n'=0}^{N-1}\abs{\frac{1}{\phi(N)} \sum_{z\in Z_N}\sum_{l_{d_2+1}\in\Integer\setminus\{0\}}\frac{\ee_{l_{d_2+1}}(y_n^{(d_2+1)}(z)
- y_{n'}^{(d_2+1)}(z))}{\abs{l_{d_2+1}}^{\alpha_2}}}\\
&=&\frac{1}{N^2}\sum_{n=0}^{N-1}\sum_{n'=0}^{N-1}\abs{\frac{1}{\phi(N)} \sum_{z\in Z_N}
\sum_{l_{d_2+1}\in\Integer\setminus\{0\}}\frac{\ee_{l_{d_2+1}}(y_{n-n'}^{(d_2+1)}(z))}{\abs{l_{d_2+1}}^{\alpha_2}}}\\
&=&\frac{1}{N}\sum_{n=0}^{N-1}\abs{\frac{1}{\phi(N)} \sum_{z\in Z_N}
\sum_{l_{d_2+1}\in\Integer\setminus\{0\}}\frac{\ee_{l_{d_2+1}}(y_n^{(d_2+1)}(z))}{\abs{l_{d_2+1}}^{\alpha_2}}}\\
&=&\frac{1}{N}\sum_{n=1}^{N}\abs{\frac{1}{\phi(N)} \sum_{z\in Z_N}
\sum_{l_{d_2+1}\in\Integer\setminus\{0\}}\frac{\ee_{l_{d_2+1}}(y_n^{(d_2+1)}(z))}{\abs{l_{d_2+1}}^{\alpha_2}}},\\
\end{eqnarray*}
where we again used that the inner sum in the second line always has the same value. We now write
\begin{eqnarray*}
 \Sigma_D=\frac{1}{N}\sum_{n=1}^{N}\abs{\frac{1}{\phi(N)} \sum_{z\in Z_N}
\sum_{l_{d_2+1}\in\Integer\setminus\{0\}}\frac{\ee^{2\pi\icomp nzl_{d_2+1}/N}}{\abs{l_{d_2+1}}^{\alpha_2}}}.
\end{eqnarray*}
Applying \cite[Lemmas 2.1 and 2.3]{KJ02} again yields
$$\Sigma_D\le \frac{4\zeta(\alpha_2)}{N},$$
which, in turn, implies 
$$\theta_{(d_1,d_2+1)}(z_{d_2+1})\le 
\frac{\gamma_{d_2 +2}^{(1)}4\zeta(\alpha)}{N}\left[\prod_{j=1}^{d_1}\left(1+\gamma_j^{(1)}\mu(\alpha_1)\right)\right]
\left[\prod_{j=1}^{d_2}\left(1+\gamma_j^{(2)}2\zeta(\alpha_2)\right)\right].$$
Combining the latter result with Equation \eqref{eqindassumption2}, 
we obtain
\begin{eqnarray*}
\lefteqn{e^2_{(d_1,d_2+1),\bsalpha,\bsgamma}(\bsg_{d_1}^*,(\bsz_{d_2}^*,z_{d_2+1}))}\\ & \le &
\frac{2}{N} 
\left[\prod_{j=1}^{d_1}\left(1+2\gamma_j^{(1)}\mu (\alpha_1)\right)\right] 
\left[\prod_{j=1}^{d_2+1}\left(1+\gamma_j^{(2)}4\zeta (\alpha_2)\right)\right].
\end{eqnarray*}
The result follows.
\end{proof}


\end{document}